\documentclass[12pt]{amsart}
\usepackage[latin1]{inputenc}
\usepackage{amssymb, amsmath, amscd, mathrsfs,marvosym, psfrag,color,a4} 

\input xy
\xyoption{all}
\DeclareMathAlphabet{\mathpzc}{OT1}{pzc}{m}{it}

\newtheorem{theorem}{Theorem}[section]

\newtheorem{proposition}[theorem]{Proposition}

\newtheorem{conjecture}[theorem]{Conjecture}

\newtheorem{lemma}[theorem]{Lemma}

\theoremstyle{definition}
\newtheorem{definition}[theorem]{Definition}

\numberwithin{equation}{section}
\theoremstyle{remark}
\newtheorem{remark}[theorem]{Remark}

\def\le{\leqslant}

\newcommand{\CG}{{\mathcal G}}

\newcommand{\CJ}{{\mathcal J}}

\newcommand{\CO}{{\mathcal O}}

\newcommand{\CS}{{\mathcal S}}

\newcommand{\CW}{{\mathcal W}}

\newcommand{\CZ}{{\mathcal Z}}

\newcommand{\SB}{{\mathscr B}}

\newcommand{\SF}{{\mathscr F}}

\newcommand{\SM}{{\mathscr M}}

\newcommand{\fh}{{{\mathfrak h}}}

\newcommand{\fg}{{{\mathfrak g}}} 
\newcommand{\fb}{{{\mathfrak b}}}

\newcommand{\fn}{{{\mathfrak n}}}

\newcommand{\fs}{{{\mathfrak s}}}

\newcommand{\fA}{{{\mathfrak A}}} 
 
\newcommand{\fC}{{{\mathfrak C}}} 
 
\newcommand{\fE}{{{\mathfrak E}}}

\newcommand{\fM}{{{\mathfrak M}}}

\newcommand{\fT}{{{\mathfrak T}}} 
\newcommand{\fU}{{{\mathfrak U}}}

\newcommand{\fX}{{{\mathfrak X}}} 
\newcommand{\fY}{{{\mathfrak Y}}}

\newcommand{\fhd}{\fh^\star}

\newcommand{\hfg}{{\widehat\fg}}

\newcommand{\tCO}{\widetilde{\CO}}

\newcommand{\tDV}{\widetilde{\DV}}

\newcommand{\tP}{{\widetilde{P}}}
\newcommand{\tQ}{\widetilde{Q}}

\newcommand{\tS}{{\widetilde{S}}}

\newcommand{\tZ}{{\widetilde{Z}}}

\newcommand{\tDelta}{\widetilde{\Delta}}

\newcommand{\DA}{{\mathbb A}}
\newcommand{\DC}{{\mathbb C}}
\newcommand{\DP}{{\mathbb P}}
\renewcommand{\DH}{{\mathbb H}}

\newcommand{\DZ}{{\mathbb Z}}

\newcommand{\DN}{{\mathbb N}}

\newcommand{\DV}{{\mathbb V}}

\newcommand{\DW}{{\mathbb W}}

\newcommand{\ch}{{\operatorname{ch}\, }}

\newcommand{\Spec}{{\operatorname{Spec}}}

\newcommand{\Hom}{{\operatorname{Hom}}}

\newcommand{\catmod}{{\operatorname{-mod}}}

\newcommand{\Sym}{{\operatorname{Sym}}}

\newcommand{\im}{{\operatorname{im}\,}}

\newcommand{\rk}{{{\operatorname{rk}}}}

\newcommand{\ol}{\overline}

\newcommand{\id}{{\operatorname{id}}}

\newcommand{\re}{{\operatorname{re}}}

\newcommand{\linie}{{\,\text{---\!\!\!---}\,}}

\newcommand{\IC}{{\operatorname{IC}}}

\newcommand{\Grass}{{\operatorname{Grass}}}

\newcommand{\comment}[1]{}

\newcommand{\lgl}{\langle}
\newcommand{\rgl}{\rangle}
\begin{document}

\pagenumbering{arabic}
\title[]{Localization of IC-complexes on  Kashiwara's flag scheme and representations of Kac--Moody algebras} \author[]{Giovanna Carnovale, Francesco Esposito, Peter Fiebig}
\begin{abstract} We study equivariant localization of intersection cohomology complexes on Schubert varieties in Kashiwara's flag manifold. Using moment graph techniques we establish a link to the representation theory of Kac-Moody algebras and give a new proof of the Kazhdan--Lusztig conjecture for blocks containing an antidominant element.
\end{abstract}

\maketitle
\section{Introduction}

We study constructible sheaves, in particular intersection cohomology complexes, on Kashiwara's version $\fX$ of the flag variety associated to a symmetrizable Kac-Moody algebra $\fg$.  We are particularly interested in the connection between the geometry of Schubert varieties in $\fX$ and the structure  of category  $\CO^\vee$  of the Langlands dual Kac-Moody algebra. Note that the theory of coherent sheaves on $\fX$ was studied extensively by, among others, Kashiwara, Shimozono \cite{KS} and Kumar \cite{Ku66}). But the theory of constructible sheaves on $\fX$, in particular with coefficients in prime characteristic, seems not to have been studied extensively yet. 

Kashiwara's flag scheme (sometimes denoted by $G/B^-$) is, in contrast to $G/B$, which  is probably more often in the focus of current research, a scheme and not merely an ind-scheme. It is stratified, but the strata are not of finite type (note that we neglect the case of finite dimensional Kac-Moody algebras, where the two versions coincide). Every Schubert variety contains infinitely many strata. However, every  finite open union of strata is in fact an $\DA^{\infty}$-fibration over a variety of finite type. Hence we can transfer most of the cohomological machinery used in geometric representation theory to these open sets. In particular, we study the localization of the torus  equivariant intersection cohomology complexes on Schubert varieties on the associated moment graph. 

Note that the moment graphs associated to $G/B$ and to $G/B^-$ coincide as labeled graphs, yet they carry  opposite partial orders on their sets of vertices. Now, a version of Soergel's structure functor allows us to relate Verma multiplicities of projective objects in the category $\CO^\vee$ to the ranks of certain sheaves on the moment graph, and hence to the dimension of stalks of certain IC-sheaves on Kashiwara's flag scheme. The case of $G/B$ yields information about blocks in positive level (i.e., blocks that contain a dominant weight), cf. \cite{F07}. In the present paper we show that Kashiwara's version   $G/B^-$ gives us information about blocks in negative level, i.e. blocks  that contain an antidominant element. As the ranks of the restriction of the IC-sheaves  on Schubert varieties in $G/B^-$ are known (by work of Kashiwara and Tanisaki), we obtain a new proof of the Kazhdan--Lusztig conjecture for regular blocks of category $\CO^\vee$ that contain an antidominant weight. 

In the paper  \cite{SVV} the reader can find a similar study of the relation between the geometry of Kashiwara's flag scheme and the representation theory of Kac-Moody algebras in the affine case (with a particular focus on Koszul duality). 
 In contrast to \cite{SVV} we tried to set up the theory in such a way that a large part generalizes to coefficients in positive characteristics. We succeeded only in parts, as  we were unable to prove the existence of equivariant parity sheaves on $\fX$.

\subsection*{Acknowledgements} This research was supported by Project  34672 ``Parity Sheaves on Kashiwara's flag manifold'', funded by  the MIUR-DAAD Joint Mobility Program  2018/2019 and by BIRD179758/17 Project ``Stratifications in algebraic groups, spherical varieties, Kac-Moody algebras and Kac-Moody groups'' funded by the University of Padova. 

\section{Kashiwara's flag scheme}\label{sec-KasflagScheme}

Let $\fg=\fn^-\oplus\fh\oplus\fn$ be a triangularized complex symmetrizable Kac-Moody algebra, $R\subset\fhd$ the set of roots of $\fg$ and $R^+\subset R$ the set of roots of $\fn$. 
For $\alpha\in R$ we denote by $\fg_\alpha\subset\fg$ the corresponding root space. Denote by $\hfg$ the completion of $\fg$ in the positive root direction, i.e.  $\hfg=\fn^-\oplus\fh\oplus\prod_{\alpha\in R^+}\fg_{\alpha}$. Denote by $\CW$ the Weyl group of $\fg$, by $l\colon \CW\to\DN$ the length function (that comes from the choice of simple reflections induced by the above triangular decomposition), and by $\le$ the Bruhat order on $\CW$.

\subsection{(Pro-)unipotent groups}\label{sec:pro_uni}
For a finite dimensional complex nilpotent Lie algebra $\fs$ we denote by $\exp(\fs)$ the associated unipotent group. Recall that $\exp(\fs)=\fs$ as a set and that the group structure is given by the Campbell-Hausdorff formula (cf. \cite[Chap. IV \S 2, no.4]{DG}).

We say that a subset $S$ of $R^+$ is {\em additively closed} if  $(S+S)\cap R^+\subset S$. For an additively closed subset $S$ of $R^+$ the subvectorspace  $\fn_S:=\bigoplus_{\alpha\in S}\fn_\alpha$ is a subalgebra of $\fn$. More generally, if  $S^\prime\subset S$ satisfies $(S+S^\prime)\cap R^+\subset S^\prime$, then $\fn_{S^\prime}$ is an ideal in $\fn_S$. If moreover $S\setminus S^\prime$ is finite, then $\fn_S/\fn_{S^\prime}$ is a finite dimensional nilpotent Lie algebra. So we can define
\begin{equation}\label{eq:limit}
\fU_S:=\varprojlim_{S^\prime} \exp(\fn_S/\fn_{S^\prime}),
\end{equation}
where $S^\prime$ ranges over  cofinite  subsets of $S$ with $(S^\prime+S)\cap R^+\subset S^\prime$. We set $\fU:=\fU_{R^+}$ and we consider $\fU_{S^\prime}$ as a subscheme in $\fU_S$ for cofinite inclusions $S^\prime\subset S$.  .

We have the following natural identifications of $\DC$-schemes 
\begin{align}\label{eq:unipotgrp}
\fU_S=\Spec(\Sym(\bigoplus_{\alpha\in S}\fg_\alpha^\ast))=\prod_{\alpha\in S}\fg_\alpha.
\end{align}
So $\fU_S$ is a prounipotent affine group scheme. 
Denote by $\fT$ the algebraic torus with Lie algebra $\fh$. The action of $\fT$ on $\fn_S$ induces an action  on $\fU_S$. Then the identifications \eqref{eq:unipotgrp} are  compatible with the $\fT$-actions.   Note that the inclusion $\fU_S\subset\fU$ splits and we have
$$
\fU=\fU_S\times \bigoplus_{\alpha\in R^+\setminus S} \fg_{\alpha}.
$$

For a finite subset  $\CJ$ of $\CW$  define  
\begin{align*}
S_\CJ&:=\{\alpha\in R^+\mid {x^{-1}}(\alpha)\in R^+\text{ for all $x\in\CJ$}\}\\
&=R^+\cap\bigcap_{x\in\CJ} x(R^+).
\end{align*} 
Then $S_{\CJ}$ is a cofinite additively closed subset of $R^+$, and for $\CJ^\prime\subset\CJ$ we have a reverse inclusion $S_{\CJ}\subset S_{\CJ^\prime}$. Set   
\begin{align*}
\fU_\CJ&:=\fU_{S_{\CJ}}\\
&=\prod_{\alpha\in R^+\cap\bigcap_{x\in\CJ}x(R^+)}\fg_{\alpha}.
\end{align*}
This  is an algebraic  subgroup scheme of $\fU$ of finite codimension, and for $\CJ^\prime\subset\CJ$ we have $\fU_\CJ\subset\fU_{\CJ^\prime}$. Clearly, $\fU_{\{e\}}=\fU$.

\subsection{Kashiwara's flag scheme}

In \cite{Kas} Kashiwara constructed a separated $\DC$-scheme $\mathfrak X$ of infinite type inside the $\DC$-scheme $\Grass(\hfg)$ of subvector spaces of  $\hfg$ as follows. Denote by $x_0$ the point in $\Grass(\hfg)$ corresponding to the subspace $\fn^-$. The actions of $\fU$ and $\fT$ on $\hfg$ give rise to actions on $\Grass(\hfg)$.  Note that $x_0$ is a $\fT$-fixed point.

\begin{lemma}[{\cite[Lemma 4.4.1]{Kas}}] \label{lemma-emb}The morphism $\fU\to \Grass(\hfg)$, $g\mapsto gx_0$, is an embedding. 
\end{lemma}

Kashiwara also constructs an action of the braid group $\widetilde{\CW}$ associated with $R$ on $\Grass(\hfg)$. For $\tilde{w}\in\widetilde{\CW}$ we consider the set $\tilde{w}(\fU x_0)$. Since $\fU x_0$ is $\fT$-stable,  $\tilde{w}(\fU x_0)$ depends only on the Weyl group element  $w$ corresponding to $\tilde w$. So we can define $\fA^w:=w(\fU x_0)$ for any $w\in\CW$. This is a $\fT$-stable subscheme of $\Grass(\hfg)$.

\begin{lemma}\label{lemma-Uacts}
Let $w\in\CW$ and suppose that $\CJ\subset\CW$ is finite and contains $w$.   Then $\fA^w$ is stable under the action of $\fU_{\CJ}$, and $\fU_{\CJ}$ acts freely on $\fA^w$. Moreover, there is a commutative diagram
\begin{align*}
\xymatrix{
\fA^w\ar@{->>}[r]\ar[d]^{\sim}&\fA^w/\fU_{\CJ}\ar[d]^{\sim}\\
\prod_{\alpha\in w(R^+)}\fg_{\alpha}\ar@{->>}[r]&\bigoplus_{\alpha\in w(R^+)\setminus S_{\CJ}}\fg_{\alpha},
}
\end{align*}
of $\fT$-schemes.\end{lemma}
 
\begin{proof} Note that  $S_{\CJ}\subset w(R^+)$ since $w\in\CJ$, and the complement is finite, as $\CJ$ is finite.
 For $w=e$,  the action of $\fU$ on $\fA^e$ is  principal homogeneous by Lemma \ref{lemma-emb}. We hence have an identification $\fA^e\cong\prod_{\alpha\in R^+}\fg_{\alpha}$ of $\fT$-schemes. For any $w\in\CW$ we hence obtain an identification $\fA^w\cong\prod_{\alpha\in w(R^+)}\fg_\alpha$ of $\fT$-schemes. Moreover, the  ``subgroup''  $\fU\cap w\fU w^{-1}$ leaves $\fA^w=w\fA^e$ stable and acts freely. Since we haven't defined a Kac--Moody group containing $\fU$, the notation $\fU\cap w\fU w^{-1}$ is only an intuitive notation for the subgroup $\fU_{S_{\{w\}}}$ of $\fU$. But this subgroup contains  $\fU_\CJ$ since $\CJ$ contains $w$. As $\fU_{\CJ}=\prod_{\alpha\in R^+\cap\bigcap_{x\in\CJ}x(R^+)}\fg_{\alpha}$, we obtain an identification $\fA^w/\fU_{\CJ}\cong\bigoplus_{\alpha\in w(R^+)\setminus S_{\CJ}}\fg_\alpha$ as claimed. The identifications thus obtained clearly fit into the above commutative diagram. 
\end{proof}

Kashiwara defines the flag variety in \cite[Definition 4.5.6]{Kas} as 
$$
\mathfrak X:=\bigcup_{w\in\CW} \fA^w.
$$
Each $\fA^w$ is an open affine subscheme in $\mathfrak X$.

\subsection{Schubert varieties}

We set $\mathfrak B:=\fT\ltimes \fU$. For $w\in\CW$ define
$$
\mathfrak  C^w:={\mathfrak B}wx_0\subset \mathfrak X
$$
and 
$$
\mathfrak X^w:=\ol{\fC^w}\subset\mathfrak X. 
$$
$\mathfrak  C^w$  is a locally closed subscheme  of $\mathfrak X$ and 
 $\mathfrak X^w$ is a closed subvariety of $\mathfrak X$  of  codimension $l(w)$ in $\mathfrak X$ (cf. \cite[Corollary 4.5.8]{Kas}).
\begin{proposition}\label{prop-quotBruhatcells} Let $w\in\CW$.
\begin{enumerate}
\item We have  $\fC^w\subset\fA^w$ and $\ol{\fC^w}=\bigsqcup_{w\le w^\prime} \fC^{w^\prime}$.
\item There is a commutative diagram of  $\fT$-schemes
\begin{align*}
\xymatrix{
\fC^w\ar@{^{(}->}[r]\ar[d]^{\sim}&\fA^w\ar[d]^{\sim}\\
\prod_{\alpha\in R^+\cap w(R^+)}\fg_{\alpha}\ar@{^{(}->}[r]&\prod_{\alpha\in w(R^+)}\fg_{\alpha}.
}
\end{align*}
\item Suppose that $\CJ$ is a finite subset in $\CW$ that contains $w$. Then there is a commutative diagram of $\fT$-schemes
\begin{align*}
\xymatrix{
\fC^w/\fU_{\CJ}\ar@{^{(}->}[r]\ar[d]^{\sim}&\fA^w/\fU_{\CJ}\ar[d]^{\sim}\\
\bigoplus_{\alpha\in R^+\cap(w(R^+)\setminus S_{\CJ})}\fg_{\alpha}\ar@{^{(}->}[r]&\bigoplus_{\alpha\in w(R^+)\setminus S_{\CJ}}\fg_{\alpha}.
}
\end{align*}
\end{enumerate}
\end{proposition}

\begin{proof} Statement (1) is \cite[Lemma 4.5.7]{Kas} and  \cite[Proposition 4.5.11]{Kas}. Statement (2) is again \cite[Lemma 4.5.7]{Kas}  (but note that Kashiwara states a non-$\fT$-equivariant version in loc.cit.). Statement (3) follows from (2) and Lemma \ref{lemma-Uacts} and the fact that $\fC^w$ is stable under the action of $\fU$, so in particular under the action of $\fU_\CJ$.  
\end{proof}

\subsection{Approximations of Schubert varieties and finite dimensional quotients}
One way to understand the geometry of the Schubert varieties $\fX^w$ is to study open subschemes that are fibrations over a finite dimensional, complex, separated scheme. In order to do this, we need the following definition.

\begin{definition} We say that $\CJ\subset \CW$ is  {\em open} if $y\le x$ and $x\in\CJ$ imply $y\in\CJ$.
\end{definition}

Let  $\CJ$ be a finite open subset   of $\CW$. We define $\mathfrak X(\CJ):=\bigcup_{w\in\CJ} \fA^w$. This is an open subscheme of $\mathfrak X$. By Lemma \ref{lemma-Uacts} it is acted upon by $\fU_{\CJ}$. We have $\mathfrak X(\CJ)\subset\mathfrak X(\CJ^\prime)$ for $\CJ\subset\CJ^\prime$, and $\mathfrak X$ is the union of all $\mathfrak X(\CJ)$ with $\CJ$ ranging over finite open subsets of $\CW$.
For $w\in\CJ$ we set
$$
\mathfrak X^w(\CJ):=\mathfrak X^w\cap\mathfrak X(\CJ).
$$
Then $\mathfrak X^w=\bigcup_{\CJ}\mathfrak X^w(\CJ)$, where $\CJ$ ranges over all finite open subsets of $\CW$. This is an open covering of $\mathfrak X^w$.
 
\begin{proposition}\label{prop:fibration} Let $\CJ,\CJ'\subset\CW$ be finite open subsets with $\CJ\subseteq\CJ'$ and suppose that  $w\in\CJ$.
\begin{enumerate}
\item\label{item:uno} The group $\fU_{\CJ^\prime}$ acts freely on $\mathfrak X^w(\CJ)$ and the quotient ${\mathfrak X}^w(\CJ)/\fU_{\CJ^\prime}$ is a separated $\DC$-scheme of finite type. 
\item\label{item:due} The canonical map $\pi^{\CJ}_{\CJ'}\colon\mathfrak X^w(\CJ)\to {\mathfrak X}^w(\CJ)/\fU_{\CJ^\prime}$ is an $\DA^{\infty}$-fibration. 
\item\label{item:tre} The images $\pi_{\CJ'}^{\CJ}(\fC^x)=\fC^x/\fU_{\CJ^\prime}$ for $w\le x$ and  $x\in\CJ$ yield a stratification of ${\mathfrak X}^w(\CJ)/\fU_{\CJ^\prime}$.
\item\label{item:cinque} The map  $\pi_{\CJ'}^{\CJ}$  is a ${\fT}$-equivariant fibration on strata and we have a commutative diagram of $\fT$-schemes
\begin{align*}
\xymatrix{
\fC^x\ar@{->>}[r]\ar[d]^{\sim}&\fC^x/\fU_{\CJ^\prime} \ar[d]^{\sim}\\
 \prod_{
\alpha\in R^+\cap x(R^+)}\mathfrak g_\alpha
\ar@{->>}[r]&\bigoplus_{\alpha\in R^+\cap(x(R^+)\setminus S_{\CJ^\prime})} \mathfrak g_\alpha.
}
\end{align*}
for each stratum. In particular, $\fC^x/\fU_{\CJ^\prime}$ is an affine space of finite dimension. 
\end{enumerate}
\end{proposition}
\begin{proof}
The proof of \cite[Lemma 6.1]{Ku66} shows that the $\fU_{\CJ^\prime}$-action is free, 
that  $\mathfrak X^w(\CJ)/\fU_{\CJ^\prime}$ is of finite type and \eqref{item:due}. Claim \eqref{item:tre} follows from \eqref{item:uno} and \eqref{item:due}. 

We prove separability of 
${\mathfrak X}^w(\CJ)/\fU_{\CJ^\prime}$ by showing that, for any $w_1,w_2\in\CJ$ there exists a  $\fU_{\CJ^\prime}$-invariant regular function $f$ on $\mathfrak X^w(\CJ)\cap \fA^{w_1}$
such that:
\begin{enumerate}
\item[(a)]\label{item:cond1}$\mathfrak X^w(\CJ)\cap \fA^{w_1}\cap \fA^{w_2}=\{x\in  \mathfrak X^w(\CJ)\cap\fA^{w_1}~|~f(x)\neq0\}$;
\item[(b)]\label{item:cond2} there exists a regular function $g$ on $\mathfrak X^w(\CJ)\cap \fA^{w_2}$  such that $g=0$ on the complement of ${\mathfrak X^w(\CJ)\cap \fA^{w_1}\cap \fA^{w_2}}$ and
$gf|_{\mathfrak X^w(\CJ)\cap \fA^{w_1}\cap \fA^{w_2}}=1$.
\end{enumerate} 
By \cite[Corollary 4.5.5]{Kas} we have  $\mathfrak{X}\cap \Grass_{\tau\hat{\fn}}(\fg)=\fA^\tau$ for any $\tau\in \CW$, where $\Grass_{\tau\hat{\fn}}(\fg)$ is the subscheme defined in  \cite[(2.2.2)]{Kas}.
The restriction $\varphi$  to $\mathfrak X^w(\CJ)\cap \fA^{w_1}$ of the function $f$ constructed in the proof of \cite[(2.2.4)]{Kas} satisfies (a) and (b) and, with the identification $\mathfrak{X}\cap \Grass_{w_1\hat{\fn}}(\fg)=\fA^{w_1}$,  it is  the determinant of a natural linear map $\psi$  between the finite-dimensional spaces $w_1\hat{\fn}/(w_1\hat{\fn}\cap w_2\hat{\fn})$ and $w_2\hat{\fn}/(w_1\hat{\fn}\cap w_2\hat{\fn})$. 
Any $u\in\fU_{\CJ}$ preserves $w_1\hat{\fn}$ and $w_2\hat{\fn}$,  and acts as a unipotent linear map on  $w_1\hat{\fn}/(w_1\hat{\fn}\cap w_2\hat{\fn})$ and $w_2\hat{\fn}/(w_1\hat{\fn}\cap w_2\hat{\fn})$. Then $u\cdot \varphi$ is the determinant of the composition of $\psi$ with two unipotent maps, so the regular function  $\varphi$ is also $\fU_{\CJ}$-invariant.

Finally, \eqref{item:cinque} is readily seen from the identifications in Proposition \ref{prop-quotBruhatcells}. 
 \end{proof}
For a finite open subset $\CJ$ with $w\in\CJ$
\subsection{The $1$-skeleton of the torus action}
Let $\CJ\subset\CW$ be finite and open. 
 \begin{proposition}\label{prop-skeleton}
 \begin{enumerate}
 \item Let $x\in\CJ$. Then $(\fC^x/\fU_{\CJ})^{\fT}=\{xx_0\fU_\CJ\}$. 
 \item Let $\fE\subset\fX^w(\CJ)/\fU_{\CJ}$ be a one-dimensional $\fT$-orbit. Then there is a real root $\alpha\in R^+$ such that  $\fE$ identifies with $\fg_{\alpha}^\times$ as a $\fT$-scheme. Moreoever, $\ol \fE=\fE\cup\{xx_0,s_{\alpha}xx_0\}$ for some $x\in\CW$.
 \end{enumerate} 
\end{proposition} 
\begin{proof} Note that (1) follows readily from Proposition \ref{prop:fibration}, \eqref{item:cinque}. In order to prove (2), suppose that $\fE$ is contained in $\fC^x/\fU_\CJ$. Again, Proposition \ref{prop:fibration}, \eqref{item:cinque} implies that there is $\alpha\in R^+\cap (x(R^+)\setminus S_\CJ)$ such that $\fE\cong \fg_{\alpha}^\times$ as a $\fT$-scheme. By definition of $S_\CJ$, we have $w^{-1}(\alpha)\in R^-$ for some $w\in\CJ$. 
This implies that $\alpha$ is real (cf. \cite[Proposition 5.2]{Kac}). Moreover, the unique $\fT$-fixed point $xx_0\fU_\CJ$ is in the closure of $\fE$, and $\fE\cup\{xx_0\fU_\CJ\}$ is an $\exp(\fg_\alpha)$-orbit. Hence $s_{\alpha}xx_0$ is contained in the closure of $\fE$ as well. Under the isomorphism $\fC^x/\fU_\CJ\cong \bigoplus_{\alpha\in R^+\cap(x(R^+)\setminus S_\CJ}\fg_\alpha$, the orbit $\fE$  is mapped to 
$\fg_{\alpha}^\times$. The action of $s_\alpha$ ``inverts''  this orbit and maps the fixed point $xx_0\fU_\CJ$ to $s_\alpha xx_0 \fU_\CJ$, which is contained in $\fC^{s_\alpha x}/\fU_\CJ$. But $\fg_\alpha^\times\cup\{0,\infty\}=\DP_1$ is closed.
\end{proof}

Observe that the assumption that $\fg$ is symmetrizable is needed for the last two statements.

\subsection{IC-complexes}\label{sec:conventions}
 For a $\DC$-scheme $\mathfrak S$ we denote by $S$ its set of $\DC$-points, endowed with the coarsest topology  for which all regular functions $f\colon V\to \DC$ are continuous when $\DC$ is endowed with the analytic topology, and $V$ is the set of $\DC$-points of an affine open subscheme $\mathfrak V$ of $\mathfrak S$. In this way, we obtain a topological space $X$ from Kashiwara's flag scheme $\mathfrak X$. We  also obtain spaces $X^w$, $X^w(\CJ)$, $C^w$,  etc. from $\mathfrak X^w$,  $\fX(\CJ)$, $\fC^w$, etc. They are acted upon by various topological groups $U$, $B$, $T$, etc.  corresponding to $\fU$, $\mathfrak B$,  $\fT$, etc..

Let $\CJ$ be an open and finite subset of $\CW$ and $w\in\CJ$. Denote by $\IC_{\CJ,w}$ the IC-sheaf on $X^w(\CJ)/U_\CJ$ with complex coefficients. 

Let $i_x$ be the inclusion of the fixed point $xx_0 U_\CJ$ into $X^w(\CJ)/U_{\CJ}$. 

\begin{theorem}\label{thm-KLforIC} For $x\in\CJ$ with $w\le x$ we have
$$
\DH^{2j+1}(i_x^\ast\IC_{\CJ,w})=0
\textrm{ for every j and }\sum_{j\in\DZ}\dim \DH^{2j}(i_x^\ast\IC_{\CJ,w})q^j=Q_{x,w}(q),
$$
where $Q_{x,w}\in\DZ[q]$ is the inverse Kazhdan--Lusztig polynomial associated with $x,w\in\CW$. 
\end{theorem}

\begin{proof} This is \cite[(4.8.4)]{KT}. Its proof is in loc. cit. Theorem 6.6.4 in terms of mixed Hodge modules and does not require the Kac-Moody algebra to be symmetrizable, as stated at the end of Section 4 therein. The translation of the formula in terms of perverse sheaves is obtained through the functor Rat as in \cite{Sa}.
\end{proof}

\section{Constructible sheaves on $\fX^w$}

In this section, which is not needed for the rest of the paper, we describe, following \cite{Gin-loop}, how one can establish a theory of constructible sheaves on the whole of Kashiwara's flag scheme using the fibrations over schemes of finite type as in Proposition \ref{prop:fibration}, coming from the approximations  $X(\CJ)$ indexed  by finite open subsets $\CJ$ of $\CW$. 

\subsection{A staircase on Kashiwara's manifold}


The infinite-dimensional scheme $\fX$ can be described  in terms of a staircase of finite dimensional varieties as in \cite[6.1]{Gin-loop}. 

Let $ \{\CJ_\alpha\}_{\alpha\in D}$ be a collection of finite open subsets of $\CW$, parametrized by a set $D$. Assume that the partial order $\le$ induced on $D$ by inclusion is directed.   For $\alpha,\,\beta\in D$ with $\alpha<\beta$ we have a reverse inclusion  of subgroups ${\fU}_{\CJ_\beta}\subset {\fU}_{\CJ_\alpha}$. 

Let $\alpha,\,\beta\in D$ with $\alpha<\beta$. We set
\begin{equation*}{\fM}_{\beta,\alpha}:={\mathfrak X}(\CJ_\alpha)/{\fU}_{\CJ_\beta}\textrm{ and }{\fY}(\CJ_{\alpha}):={\fM}_{\alpha,\alpha}.\end{equation*}
The quotients ${\fM}_{\beta,\alpha}$  are finite-dimensional stratified smooth algebraic varieties by Proposition \ref{prop:fibration}.

For $\alpha,\beta,\gamma$ in $D$ with $\alpha< \beta\leq\gamma$ we consider the Zariski-open embeddings $j_\gamma^{\alpha\beta}\colon {\fM}_{\gamma,\alpha}\to {\fM}_{\gamma,\beta}$ induced by the open inclusion $\mathfrak X(\CJ_\alpha)\to \mathfrak X(\CJ_\beta)$, and  for $\lambda,\,\mu,\,\nu$ in $D$ with  $\lambda\leq \mu <\nu$ we have a natural projection $p^\lambda_{\nu\mu}\colon {\fM}_{\nu,\lambda}\to {\fM}_{\mu,\lambda}$. We also set $j_\gamma^{\alpha\alpha}=\id$, $p^\lambda_{\mu\mu}=\id$. 
By Proposition \ref{prop:fibration}, the projections
$\pi^{\lambda}_{\mu}\colon \mathfrak X(\CJ_\lambda)\to {\fM}_{\mu,\lambda}$ and 
$\pi^{\lambda}_{\nu}\colon \mathfrak X(\CJ_\lambda)\to {\fM}_{\nu,\lambda}$ are principal fibrations. 
Hence, $p^\lambda_{\nu\mu}$ is a locally trivial fibration with fiber ${\fU}_{\CJ_\mu}/{\fU}_{\CJ_\nu}$.
The latter is an affine space by virtue of the identification $\fA^e\simeq\fU$, the inclusion $\fU_{\CJ_\mu}\subset\fU$ and Proposition \ref{prop-quotBruhatcells}.
By construction, $p^\lambda_{\mu\nu}$ and $j_\gamma^{\alpha\beta}$ are stratified maps.  
Clearly, for $\alpha\leq\beta\leq\gamma\leq\mu$ we have 
\begin{equation*}j_\mu^{\beta\gamma}\circ j_\mu^{\alpha\beta}=j_\mu^{\alpha\gamma}\textrm{ and }p^\alpha_{\gamma\beta}\circ p_{\mu\gamma}^\alpha=p^\alpha_{\mu\beta}.\end{equation*} In addition, for any ${\alpha}\leq\beta\leq{\gamma}\leq{\mu}$
 the restriction of the fibration $p^\beta_{\gamma\mu}$ to the open subset ${\fM}_{\gamma,\alpha}$ of the basis of the fibration ${\fM}_{\gamma,\beta}$ coincides with $p^\alpha_{\gamma\mu}$. 
 
 This way, we get a staircase ${\fY}$ of smooth stratified varieties connected by the stratified maps as here below: 
 \begin{equation}
\label{eq:staircase}
\begin{CD}
 & &&&  &&&&&&@VV  V\\
 & &&  &&& {\fM}_{\gamma,\beta} & @>{j_\gamma^{\beta\gamma}}>> &{\fY}(\CJ_\gamma) \\
  & &&&&& @VV p^{\beta}_{\gamma\beta} V&&\\ 
 &&&{\fM}_{\beta,\alpha} @>j^{\alpha\beta}_{\beta}>> &{\fY}(\CJ_\beta)&  &\\
  &&&  @VVp^\alpha_{\beta\alpha}V&  \\ 
 @>>>&{\fY}(\CJ_\alpha) &&&
\end{CD}
\end{equation} 
 
 \subsection{Constructible complexes on the staircase ${\fY}$}\label{sec:complexes_staircase}

Recall the conventions on notation from Section \ref{sec:conventions}. We now define the category $D^b_c(Y,k)$ of constructible complexes on the staircase $Y$ of the topological spaces of  ${\mathbb C}$-points  corresponding to the quotients in ${\fY}$, see also \cite[6.2]{Gin-loop}.
 
For any $\alpha\in D$,  let $D^b_\Sigma(Y(\CJ_\alpha),k)$ be the bounded derived category of complexes of sheaves on $Y(\CJ_\alpha)$ whose cohomology sheaves are locally constant for the stratification $\Sigma$ induced by the Bruhat decomposition on ${\mathfrak X}$ and $X$. 
 
 We define the category $D^b_{c,D}(Y,k)$ as follows: 
 \begin{itemize}
 \item Objects are given by the datum $({\mathcal F},\phi)=(({\mathcal F}_\alpha)_{ \alpha\in D}, (\phi_{\beta\alpha})_{\alpha,\beta\in D, \alpha\leq\beta})$ where ${\mathcal F}_\alpha$ is an object in $D^b_\Sigma(Y(\CJ_\alpha),k)$  for any $\alpha\in D$ and  
   $\phi_{\beta\alpha}$  for any $\alpha,\,\beta\in D$ such that ${\alpha}\leq {\beta}$ is an isomorphism
 \begin{equation}\label{eq:phi}
\phi_{\beta\alpha}\colon (j_\beta^{\alpha\beta})^*{\mathcal F}_{\beta}\to (p^\alpha_{\beta\alpha})^*{\mathcal F}_\alpha,
 \end{equation} 
 in $D^b_c({\fM}_{\beta,\alpha},k)$ satisfying the compatibility conditions:
 \begin{equation}
 \phi_{\alpha\alpha}=\id\quad\textrm{ and } \quad (p^\alpha_{\gamma\beta})^*(\phi_{\beta\alpha})\circ (j^{\alpha\beta}_\gamma)^*(\phi_{\gamma\beta})=\phi_{\gamma\alpha} 
 \end{equation}
in $\textrm{Mor}((j^{\alpha\gamma}_\gamma)^*{\mathcal F}_\gamma,(p^\alpha_{\gamma\alpha})^*{\mathcal F}_\alpha)$ of $D^b_c({\fM}_{\gamma,\alpha},k)$, 
 for every $\alpha,\,\beta,\,\gamma$ such that $\alpha\leq\beta\leq\gamma$.
 \item Morphisms between objects $({\mathcal F},\phi)=(({\mathcal F}_\alpha)_{\alpha\in D}, (\phi_{\beta\alpha})_{\alpha,\beta\in D, \ \alpha\leq\beta})$ and $({\mathcal G},\psi)=(({\mathcal G}_\alpha)_{\alpha\in D}, (\psi_{\beta\alpha})_{\alpha,\beta\in D,\  \alpha\leq\beta})$ are given by the data $((f_\alpha)_{\alpha\in D})$ where $f_\alpha\colon {\mathcal F}_\alpha\to {\mathcal G}_\alpha$ for $\alpha\in D$ are morphisms in $D^b_c(Y(\CJ_\alpha),k)$ satisfying the commutativity condition:
 \begin{equation}
\label{eq:morphisms}
\begin{CD}
( j^{\alpha\beta}_\beta)^*{\mathcal F}_\beta&@>{\phi_{\beta\alpha}}>> & (p^\alpha_{\beta,\alpha})^*{\mathcal F}_\alpha\\
 @VV (j^{\alpha\beta}_\beta)^*f_\beta V&@VV (p_{\beta\alpha}^\alpha)^*f_\alpha V\\
 ( j^{\alpha\beta}_\beta)^*{\mathcal G}_\beta&@>{\psi_{\beta\alpha}}>> & (p^\alpha_{\beta,\alpha})^*{\mathcal G}_\alpha
\end{CD}
\end{equation} 
 \end{itemize}

From now on we restrict to the special family $D={\mathbb N}$ and $\CJ_n=\{w\in\CW~|~l(w)\leq n\}$ for $n\geq0$ and the corresponding category $D^b_{c,{\mathbb N}}({Y},k)$.
In this case we write $j_n$,  $p^n$ and $\phi_n$, respectively instead of $j_n^{n-1,n}$,  $p^{n+1,n}_n$ and $\phi_{n+1,n}$, respectively. Any sequence $(\phi_{n,m})_{n,m\in{\mathbb N}, n> m}$ satisfying \eqref{eq:phi} is completely determined by the sequence $(\phi_n)_{n\in{\mathbb N}}$. One can prove that $D^b_{c,D}(Y,k)$ and  $D^b_{c,{\mathbb N}}(Y,k)$ are equivalent for any choice of $D$.

 
Let ${\mathcal P}({Y}(\CJ_n),k)$ be the category of perverse sheaves on ${Y}(\CJ_n)$ with respect to the induced Bruhat stratification and middle perversity shifted so that all nonzero cohomology occurs in nonnegative degree, as in \cite[8.2]{FW}.  
 We define the full subcategory ${\mathcal P}({Y},k)$ of $D^b_{c,{\mathbb N}}({Y},k)$ of perverse sheaves on ${Y}$.   Objects are those $({\mathcal F},\phi)=(({\mathcal F}_n)_{n\in {\mathbb N}}, (\phi_{n+1,n})_{n\in{\mathbb N}})$ in $D^b_{c,{\mathbb N}}({Y},k)$ such that each ${\mathcal F}_n$ is an object in ${\mathcal P}({ Y}(\CJ_n),k)$. 
It can be verified that it is an abelian category.  

\begin{proposition}Let $m\in{\mathbb N}$. Any object ${\mathcal F}_m$ in ${\mathcal P}({Y}(\CJ_m),k)$
 extends to an object in ${\mathcal P}({Y},k)$ giving a functor
$F_m\colon {\mathcal P}({ Y}(\CJ_m),k)
\to {\mathcal P}({ Y},k)$.   
\end{proposition}
\begin{proof}Given ${\mathcal F}_m$, we construct the sequence $({\mathcal F}_n)_{n\in{\mathbb N}}$ inductively as follows. We set 
\begin{align*}
&{\mathcal F}_{m+1}:=(j_{m+1})_{!*}(p^m)^*{\mathcal F}_{m}, &&{\mathcal F}_{m-1}:=(p^{m+1})_{*}(j_m)^*{\mathcal F}_{m}, \\
&{\mathcal F}_{m+l}:=(j_{m+l})_{!*}(p^{m+l-1})^*{\mathcal F}_{m+l-1}&&\textrm{ for any $l\geq 1$}\\
&{\mathcal F}_{m-l}:=(p^{m-l})_{*}(j_{m-l+1})^*{\mathcal F}_{m-l+1} &&\textrm{ for any $1\leq l\leq m$}.
\end{align*} 
The sequence  $(\phi_n)_{n\in{\mathbb N}}$ is given using the adjunction $(p^n)^*(p^n)_*\to \id$ for $n\leq m-1$ and the natural isomorphisms $(p^{n-1})^*{\mathcal F}_{n-1}\simeq (j_n)^{*}(j_n)_{!*}(p^{n-1})^*{\mathcal F}_{n-1}$ for $n\geq m$ coming from the extension property of perverse sheaves.  It is straightforward to verify that $(({\mathcal F}_n)_{n\in{\mathbb N}}, (\phi_n)_{n\in{\mathbb N}})$ gives an object in ${\mathcal P}({Y},k)$ and that this construction is functorial.
\end{proof}

Let $w\in\CW$ and let $n\geq l(w)$. Then $O_{w,n}:=C^w/U_{\CJ_n}$ is a stratum in ${X}(\CJ_n)/U_{\CJ_n}$ and it is an affine space by Proposition \ref{prop:fibration}.  The object  $\IC(X^w,k)$ in ${\mathcal P}({Y},k)$ can be defined in terms of the staircase ${Y}$, by setting $\IC(X^w,k):=F_n \IC(\overline{O_{w,n}},k)$,
where $\IC(\overline{O_{w,n}},k)$ is the $\IC$-complex with indices normalised to ensure that it is an object in ${\mathcal P}({Y}(\CJ_n),k)$.
One can prove that it does not depend on the choice of $n$, that it is simple in ${\mathcal P}({Y},k)$ and that all simple objects in ${\mathcal P}({Y},k)$ are obtained this way. 


We can also define parity sheaves on $Y$ as those objects  $(({\mathcal F}_n)_{n\in{\mathbb N}}, (\phi_n)_{n\in{\mathbb N}})$  in $D^b_{c,{\mathbb N}}({ Y},k)$ such that ${\mathcal F}_n$ is a parity sheaf for every $n\in{\mathbb N}$ with respect to the constant pariversity. 

\begin{proposition}Parity sheaves on $Y$ exist and when $k={\mathbb C}$ they are the $\IC$-sheaves $\IC_{\CJ_n,w}$.
\end{proposition}
\begin{proof}
The first statement follows  from \cite[Corollary 2.20]{JMW} because strata are contractible. The second one follows from Theorem 2.1 and \cite[Theorem 2.12]{JMW}.
\end{proof}

 \section{Moment graphs and equivariant cohomology}
 
 The link between the topology of Kashiwara's flag scheme and representation theory that we utilize is given by equivariant cohomology and moment graphs. 
 
 \subsection{The associated moment graph}\label{sec-assocmomgra}
 Let $\CJ\subset\CW$ be a finite open set. We denote by $Y(\CJ):=X(\CJ)/U_\CJ$ the topological space associated with the (finite dimensional) variety $\fX(\CJ)/\fU_\CJ$.   The corresponding moment graph $\CG(\CJ)$ is given as follows.  Its set of vertices is $Y(\CJ)^T$, the set of $T$-fixed points, and  $x,y\in Y(\CJ)^T$ are connected by an edge if and only if $x\ne y$ and there is a one-dimensional $T$-orbit $E$ such that $\ol E=E\cup\{x,y\}$.   Then the edge $E$ is homeomorphic to $\fg_\alpha^\times$ as a $T$-space for some $\alpha\in R^+$ and we set $l(E)=\alpha$ and call this the {\em label} of $E$ (cf. Proposition \ref{prop-skeleton} (2)).  Each Bruhat cell $\widetilde{C^x}=C^x/U_\CJ$ contains a unique fixed point  (cf. Proposition \ref{prop-skeleton} (1)) and the closure relation on Bruhat cells yields a partial order on the set of vertices: We write $x\le y$ if $\widetilde{C^y}$ is contained in the closure of $\widetilde{C^x}$ . 

Proposition \ref{prop-skeleton} shows that we have a canonical identification $Y(\CJ)^T=\CJ$. Moreover, $x,y\in\CJ$ are connected if and only if there exists $\alpha\in R^+$ such that $y=s_\alpha x$, and the edge is labeled by $\pm\alpha$. Finally, the partial order on the vertices coincides with the Bruhat order, by Proposition \ref{prop-quotBruhatcells}. 
Note that  two connected vertices are comparable by Proposition \ref{prop-skeleton}.  

Denote by  $S$ the symmetric algebra over the vector space $\fhd$. We consider $S$ as a graded algebra with $\fhd\subset S$ being the homogeneous component of degree $2$.
The {\em structure algebra}  of the moment graph $\CG(\CJ)$ is
$$
\CZ(\CJ):=\left\{(z_x)\in\prod_{x\in\CJ} S\left|
\begin{matrix}
\text{ $z_x\equiv z_{s_\alpha x}\mod \alpha$}\\
\text{  for all $x\in\CJ$, $\alpha\in R^+$ such that $s_\alpha x\in\CJ$}
\end{matrix}
\right\}.
\right.
$$
Coordinatewise addition and multiplication makes $\CZ(\CJ)$ into an
$S$-algebra. 


\subsection{The Braden--MacPherson sheaves}\label{sec-BMPsheaves}
Of particular importance is the notion of a sheaf on a moment graph. 
\begin{definition}
A {\em sheaf} $\SM$ on the moment graph $\CG(\CJ)$ is given by the following data:
\begin{itemize}
\item an $S$-module $\SM^x$ for any vertex $x\in\CJ$,
\item an  $S$-module $\SM^E$ with $l(E)\SM^E=0$ for all edges $E$ of $\CG(\CJ)$.
\item a homomorphism $\rho_{x,E}\colon \SM^x\to \SM^E$ of  $S$-modules for any vertex $x$ lying on the edge $E$.
\end{itemize}
\end{definition}
Let $\SM$ be a sheaf on $\CG$. For an open subset $\CJ^\prime$ of $\CJ$ we
define the {\em space of sections} of $\SM$ over $\CJ^\prime$ by
$$
\Gamma(\CJ^\prime,\SM):=\left\{(m_x)\in\prod_{x\in\CJ^\prime} \SM^x\left|\begin{matrix}
      \rho_{x,E}(m_x)=\rho_{y,E}(m_y) \\
\text{ for all edges $E\colon x\linie y$}\\
\text{ with $x,y\in\CJ^\prime$}
\end{matrix}\right.\right\}.
$$
Coordinatewise multiplication makes $\Gamma(\CJ^\prime,\SM)$ into a
$\CZ(\CJ^\prime)$-module.  We call the space $\Gamma(\SM):=\Gamma(\CJ,\SM)$ the space of {\em global sections}. For two open subsets $\CJ^{\prime\prime}\subset \CJ^\prime$  of $\CJ$ the canonical projection $\prod_{x\in \CJ^\prime}\SM^x\to \prod_{x\in \CJ^{\prime\prime}}\SM^x$ induces a {\em restriction map} $\Gamma(\CJ^\prime,\SM)\to\Gamma(\CJ^{\prime\prime},\SM)$.

Of particular importance is the following family of sheaves on $\CG(\CJ)$.

\begin{theorem}[{\cite[Section 1.4]{BMP},\cite[Definition 6.3 \& Theorem 6.4]{FW}}] Let $x\in\CJ$ be a vertex. There is an up to isomorphism unique sheaf $\SB(x)$ on $\CG(\CJ)$ with the following properties.
\begin{enumerate}
\item $\SB(x)$ is indecomposable.
\item For any $w\in\CJ$, the $S$-module $\SB(x)^w$ is (graded) free, $\SB(x)^w\ne(0)$ implies $x\le w$ and $\SB(x)^x\cong S$.
\item For any open subset $\CJ^\prime$ of $\CJ$, the restriction homomorphism $\Gamma(\CJ,\SB(x))\to\Gamma(\CJ^\prime,\SB(x))$ is surjective.
\item For any $w\in\CJ$, the homomorphism $\Gamma(\SB(x))\to\SB(x)^w$, $(z_y)\mapsto z_w$, is surjective.
\end{enumerate}
\end{theorem}
$\SB(x)$ is called the {\em Braden--MacPherson sheaf} associated with $x$. Recall that we consider $S$ as a graded algebra, so projective covers exist in the category of $S$-modules. This ensures that the algorithm of the ``canonical sheaf'' in \cite{BMP} works. 

\subsection{Localization of equivariant sheaves} For a locally closed inclusion $i\colon W\subset Z$ of topological spaces and a sheaf $\SF$ on $Z$ we write $\SF_W$ for the restriction $i^\ast\SF$.  Again let $\CJ$ be an open subset of $\CW$. 
 Let $\SF$ be a $T$-equivariant sheaf on $Y(\CJ)$, i.e. an object in $D^b_T(Y(\CJ),\DC)$.  We will now associate a sheaf $\DW(\SF)$ on $\CG(\CJ)$ to $\SF$. For a vertex $x$ (i.e., a $T$-fixed point in $Y(\CJ)$) we set  $\DW(\SF)^x:=\DH_T^\bullet(\SF_x)$, and for an edge $E$ (i.e., a one-dimensional $T$-orbit in $Y(\CJ)$) we set  $\DW(\SF)^E:=\DH_T^\bullet(\SF_E)$. Then $l(E)\DW(\SF)^E=\{0\}$ by \cite[Lemma 3.1]{FW}.

Suppose that the $T$-fixed point $x$ is contained in the closure of the one dimensional orbit $E$. Then the restriction homomorphism
$$
\DH_T^\bullet(\SF_{E\cup\{x\}})\to \DH_T^\bullet(\SF_{x})
$$
is an isomorphism (cf. \cite[Proposition 2.3]{FW}), so we can define a homomorphism $\rho_{x,E}$ as the composition
$$
\DH_T^\bullet(\SF_x)\stackrel{\sim}\leftarrow \DH_T^\bullet(\SF_{E\cup\{x\}})\to \DH_T^\bullet(\SF_E).
$$
In this way we indeed obtain a sheaf $\DW(\SF)$ from $\SF$. 

\begin{remark}\label{rem-degfromglobsec}  Let $Q$ be the quotient field of $S$. The inclusion $\CZ(\CJ)\subset\bigoplus_{x\in\CJ} S$ becomes an isomorphism after applying the functor $\cdot\otimes_S Q$ (since   $\CJ$ is finite, cf. \cite{F07}). For any $\CZ(\CJ)$-module $M$ we hence obtain a canonical generic decomposition $M\otimes_S Q=\bigoplus_{x\in\CJ}(M\otimes_SQ)^x$ that is such that $z=(t_x)$ acts on the component $(M\otimes_S Q)^x$ as multiplication with $t_x$. If $\SF$ is an equivariant sheaf on $Y(\CJ)$ such that $\DH_T(\SF_x)$ is a graded free $S$-module of finite rank, then we obtain
$$
\rk_S\, \DH_T(\SF_x)=\dim_Q (\DW(\SF)\otimes_S Q)^x.
$$
\end{remark}

\subsection{Localization of equivariant IC-sheaves}
The following result is the main reason for our interest in moment graphs. Its original version was proven  in \cite{BMP} (with characteristic $0$ coefficients). It was later generalized to almost arbitrary coefficients in \cite{FW}. In these cases one has to replace the IC-sheaves with parity sheaves. 

\begin{theorem}\label{thm-ICBMP}  Let $\CJ\subset\CW$ be finite and open, and let $x\in\CJ$.  Let $\SF$ be the $T$-equivariant intersection cohomology sheaf with complex coefficients on the Schubert variety $Y^x(\CJ):=X^x(\CJ)/U_\CJ\subset Y(\CJ)=X(\CJ)/U_\CJ$ .  Then $\DW(\SF)$ is isomorphic to the Braden--MacPherson sheaf $\SB(x)$  up to a grading shift. 
\end{theorem}
Using Remark \ref{rem-degfromglobsec}  the dimensions of stalks of the Braden-MacPherson sheaves and the local cohomologies of IC-sheaves on $X$ coincide.

\section{Passage to representation theory}

In this section we want to apply the topological results above to the representation theory of complex Kac--Moody algebras and give a new proof of the Kazhdan--Lusztig conjecture on the characters of irreducible highest weight representations at negative level. 

Before we start, we need to fix some more notation. 
Denote by $\Pi\subset R^+$ the set of simple roots,  by $R^{\re}\subset R$ the set of {\em real roots}, and by $R^{\im}=R\setminus R^{\re}$ the set of imaginary roots. For a real root $\alpha$ denote by $\alpha^\vee\in\fh$ its coroot.  Denote by $\CS$ the set of simple reflections in $\CW$. Fix once and for all a {\em Weyl vector} $\rho\in\fhd$, i.e. an element with the property $\lgl\rho,\alpha^\vee\rgl=1$ for all $\alpha\in\Pi$. Then the {\em dot-action} of $\CW$ on $\fhd$ is given by $w.\lambda=w(\lambda+\rho)-\rho$. 

We denote by $\le$ the usual partial order on $\fhd$, i.e. $\mu\le\lambda$ if $\lambda-\mu$ can be written as a sum of positive roots. Recall that we assume that $\fg$ is symmetrizable, i.e. that there exists an invariant non-degenerate symmetric form $(\cdot,\cdot)\colon\fg\times\fg\to\DC$.

Recall that an element $\lambda\in\fhd$ is called
\begin{itemize}
\item {\em integral}, if $\lgl\lambda,\alpha^\vee\rgl\in\DZ$ for all $\alpha\in \Pi$,
\item {\em regular}, if  $w\in\CW$ and $w.\lambda=\lambda$ imply $w=e$,
\item {\em non-critical}, if $2(\lambda+\rho,\delta)\not\in\DZ(\delta,\delta)$ for all $\delta\in R^{\im}$,
\item {\em anti-dominant}, if $\lgl\lambda,\alpha^\vee\rgl\not\in\DZ_{\ge 0}$ for all $\alpha\in\Pi$.
\end{itemize}

\subsection{The Kazhdan--Lusztig conjecture at negative level}

For any $\lambda\in\fhd$ we denote by $L(\lambda)$ the irreducible representation of $\fg$ with highest weight $\lambda$. For $\mu\in\fhd$ denote by $\Delta(\mu)$ the Verma module of $\fg$ with highest weight $\mu$.

The Kazhdan--Lusztig conjecture in the negative level of Kac--Moody algebras is the following:

\begin{conjecture}\label{conj-KL} Suppose $\lambda$ is  non-critical,  integral, regular and anti-dominant. Let $w\in\CW$. Then   
$$
\ch L(w.\lambda)=\sum_{y\le w} (-1)^{l(w)-l(y)}P_{y,w}(1)\ch \Delta(y.\lambda),
$$
where $P_{y,w}\in\DZ[v]$ denotes the Kazhdan--Lusztig polynomial associated with $y$ and $w$ for the Coxeter system $(\CW,\CS)$. 
 \end{conjecture}
 
Note that the above is an obvious and immediate generalization of the conjecture stated in \cite{L90}  in the affine negative level case. The latter conjecture was  proven in \cite{KT95}.

 Instead of stating the character of an irreducible module in terms of Verma characters, one can also obtain an equivalent conjecture describing Verma characters in terms of irreducibles (i.e. a conjecture on Jordan--H\"older multiplicities). If we denote by $[\Delta(\mu):L(\nu)]$ the multiplicity of $L(\nu)$ in a Jordan--H\"older filtration of $\Delta(\mu)$, then Conjecture \ref{conj-KL} is equivalent to the following conjecture.

\begin{conjecture} \label{conj-KLmult} Suppose $\lambda$ is non-critical, integral, regular and anti-dominant. Let $x,y\in\CW$. Then
$$
[\Delta(x.\lambda):L(y.\lambda)]=Q_{x,y}(1),
$$
where $Q_{x,y}\in\DZ[v]$ is the {\em inverse} Kazhdan--Lusztig polynomial associated with $x$ and $y$ for the Coxeter system $(\CW,\CS)$.
\end{conjecture}

We need yet another reformulation of the Conjecture in terms of Verma multiplicities of projective objects in category $\CO$. In the next section, we recall the necessary results. 

\subsection{Category $\CO$, projectives and BGG-reciprocity}

The category $\CO$ associated with $\fg$ is the full subcategory of the category of all $\fg$-modules that contains all $M$ that satisfy the following properties.
\begin{itemize}
\item The $\fh$-action on $M$ is diagonalizable.
\item  The $\fb:=\fn\oplus\fh$-action on $M$ is locally finite.
\end{itemize}
For any $\lambda\in\fhd$, the Verma module $\Delta(\lambda)$ and its irreducible quotient $L(\lambda)$ are objects in $\CO$.

Let $\CJ$ be a subset of $\fhd$. We say that it is {\em open}, if $\lambda\in\CJ$ and $\mu\in\fhd$ with $\mu\le\lambda$ imply $\mu\in\CJ$. An open subset $\CJ$ is called {\em locally bounded}, if for all $\lambda\in\CJ$ the set $\{\mu\in\CJ\mid \lambda\le\mu\}$ is finite.

For an open subset $\CJ$ we define the {\em truncated category} $\CO^\CJ$ as the full subcategory of $\CO$ that contains all objects $M$ with the property  that $M_\mu\ne 0$ implies $\mu\in\CJ$. For example, $\Delta(\lambda)$ is contained in $\CO^\CJ$ if and only if $L(\lambda)$ is contained in $\CO^\CJ$ if and only if $\lambda\in\CJ$. 

\begin{theorem}\label{thm-BGG}\cite[Theorem 3.4.10]{FiePal} \footnote{Note that in Section 3.4. in \cite{FiePal}, certain references to  previously published results are wrong. In Proposition 3.4.8 in loc.cit. the reference should be to [6, Proposition 2.1], and in Proposition 3.4.9 in loc.cit. the reference should be to [6, Lemma 2.3].} Suppose that $\CJ$ is open and locally bounded. Then for any $\lambda\in\CJ$ there exists a projective cover $P^\CJ(\lambda)$ of $L(\lambda)$ in $\CO^{\CJ}$. It admits a Verma flag and for the multiplicities the BGG-reciprocity holds:
$$
(P^{\CJ}(\lambda):\Delta(\mu))=[\Delta(\mu):L(\lambda)]
$$
for all $\lambda,\mu\in\fhd$.
\end{theorem}
In fact, the above is proven in \cite{FiePal} under the assumption (the general assumption in that article) that $\fg$ is finite dimensional. The proof, however, does not use finite dimensionality.

\subsection{The relation to moment graph sheaves}

The main result that allows us to link the topology of Kashiwara's flag scheme to the representation theory of $\fg$ is the following.

\begin{theorem}\label{thm-repmomgra} Suppose that $\lambda$ is non-critical, integral, regular and anti-dominant. Let $\CJ\in\fhd$ be open and locally bounded. Let  $w,x\in\CW$ and assume that $w.\lambda,x.\lambda\in\CJ$. Then
$$
(P^\CJ(w.\lambda):\Delta(x.\lambda))=\rk_S \,\SB^\vee(w)^x,
$$
where $\SB^\vee(w)$ denotes the  Braden--MacPherson sheaf associated with $w$ on the moment graph $\CG^\vee$ for the Langlands dual root datum. 
\end{theorem}

\begin{proof} The proof follows closely the proof of the analogous statement in \cite{F07}, where the dominant case is treated (in which case projective objects exist in the non-truncated block of category $\CO$). We only need to rewrite the main arguments using the  truncated category. 

First, we need  a {\em deformed version} of the main objects. As a reference for the following constructions, the reader might consult the article \cite{FiePal}. Denote by ${\tS^\vee}$ the localization of $S^\vee:=S(\fh)$, the symmetric algebra of the vector space $\fh$, at the maximal ideal $S(\fh)\fh$. Then the deformed version $\tilde\CO$ of category $\CO$ is a full subcategory of the category of all $\fg$-${\tS^\vee}$-bimodules, and it contains for any $\lambda\in\fhd$ a deformed Verma module $\tDelta(\lambda)$. This is free as a ${\tS^\vee}$-module, and it satisfies $\tDelta(\lambda)\otimes_{{\tS^\vee}}\DC\cong\Delta(\lambda)$. Likewise, there is a deformed truncated projective $\tP^\CJ(\lambda)$ in the truncated category $\tilde\CO^\CJ$ for any locally bounded open subset $\CJ$ of $\fhd$ and $\lambda\in\CJ$. This object admits a deformed Verma flag, and for the multiplicities we have $(\tP^\CJ(\lambda):\tDelta(\mu))=(P^\CJ(\lambda):\Delta(\mu))$. 

Now 
denote by $\tCO_\Lambda$ the  block of the category $\tCO$ that contains the deformed Verma module $\tDelta(\lambda)$. We identify the index $\Lambda$ with the set of all $\mu$ such that $\tDelta(\mu)$ is contained in $\Lambda$. Then  $\Lambda=\CW.\lambda$. We denote by $\tCO_\Lambda^\CJ=\tCO_\Lambda\cap\tCO^\CJ$ the truncated subcategory. Then $\tDelta(\mu)$ is contained in $\tCO^{\CJ}_\Lambda$ if and only if $\mu\in\Lambda^\CJ:=\Lambda\cap\CJ$. Note that since $\CJ$ is supposed to be locally bounded and since $\Lambda$ contains a smallest element, the set $\Lambda^\CJ$ is finite. 

Now let  $\tZ^\CJ_\Lambda$ be the {\em center} of $\tCO^\CJ_\Lambda$, i.e. the endomorphism ring of the identity functor on $\tCO^\CJ_\Lambda$. By \cite[Theorem 3.6]{F03} we have an isomorphism
$$
\tZ^\CJ_\Lambda\cong\left\{\{t_\nu\}\in\bigoplus_{\nu\in\Lambda^\CJ}{\tS^\vee}\left|\begin{matrix} t_\nu\equiv t_{s_\alpha.\nu}\mod\alpha^\vee \\\text{ for all $\nu\in\Lambda^\CJ$, $\alpha\in R^{+}\cap R^{re}$}\\ \text{ with $s_\alpha.\nu\in\Lambda^\CJ$} \end{matrix}\right\}\right..
$$
This isomorphism is normalized in such a way that the element $z=(t_\nu)\in\tZ^\CJ_\Lambda$ acts on $\tDelta(\nu)$, for $\nu\in\Lambda^\CJ$, as multiplication with the scalar $t_\nu\in{\tS^\vee}$. 

Let $\CJ^\prime$ be the set of all $x\in\CW$ with the property $x.\lambda\in\Lambda^\CJ$. Then $\CJ^\prime$ is an open subset of $\CW$ and the map $\CJ^\prime\to\Lambda^\CJ$, $x\mapsto x.\lambda$, is compatible with partial orders, and a bijection.  The definition of the structure algebra $Z^\vee(\CJ^\prime)$ of the moment graph $\CG^\vee({\CJ^\prime})$ (for the dual root system), introduced in Section \ref{sec-assocmomgra}, yields an identification 
$$
\tZ_\Lambda^{\CJ}\cong \CZ^\vee(\CJ^\prime)\otimes_{S^\vee}{\tS^\vee}.
$$

Now  consider the functor $\tDV^\CJ=\Hom_{\tCO_\Lambda^\CJ}(\tP^\CJ(\lambda),\cdot)\colon\tCO^{\CJ}\to\tZ^\CJ_{\Lambda}\catmod$. Here, for any object $M$ of $\tCO_\Lambda^\CJ$ we consider $\Hom_{\tCO_\Lambda^\CJ}(\tP^\CJ(\lambda),M)$ as a $\tZ_\Lambda^\CJ$-module via the canonical action of $\tZ_\Lambda^\CJ$ on $\tP^\CJ(\lambda)$.  Denote by ${\tQ^{\vee}}$ the quotient field of ${\tS^\vee}$.  The inclusion $\tZ^{\CJ}_\Lambda\subset\bigoplus_{\nu\in\Lambda^\CJ}{\tS^\vee}$ becomes a bijection after applying the functor $\cdot\otimes_{\tS^\vee}{\tQ^{\vee}}$ (as $\Lambda^\CJ$ is finite). So for any $\tZ^\CJ_\Lambda$-module $M$ we obtain a canonial decomposition
$$
M\otimes_{\tS^\vee} {\tQ^{\vee}}=\bigoplus_{\nu\in\Lambda^\CJ} (M\otimes_{\tS^\vee}{\tQ^{\vee}})^\nu,
$$
which has the property that $z=(t_\nu)$ acts on $(M\otimes_{\tS^\vee}{\tQ^{\vee}})^\nu$ as multiplication with $t_\nu$. 

If $M$ is an object in $\tCO_\Lambda^\CJ$ that admits a Verma flag (i.e., a filtration with subquotients isomorphic to various deformed Verma modules), then $\tilde\DV^\CJ M$ admits a Verma flag as well (for a definition, see Section 4 in \cite{F07}). By the formula for the action of the center on the deformed Verma modules we have
$$
(M:\tDelta(x.\lambda))=\dim_{{\tQ^{\vee}}} (\tilde\DV^{\CJ}(M)\otimes_{\tS^\vee}{\tQ^{\vee}})^x.
$$

Finally, the main step in the proof of the Theorem was already done in \cite{F07}.  By Remark 7.6. in \cite{F07}, there is an isomorphism 
$$
\tDV^\CJ \tP^\CJ(w.\lambda)\cong \Gamma(\CJ^\prime,\SB^\vee(w))\otimes_{S^\vee}{\tS^\vee}
$$
of $\CZ^\vee(\CJ^\prime)
\otimes_{S^\vee}{\tS^\vee}$-modules. 
From the  above we obtain an isomorphism
$$
(\tDV^\CJ \tP^{\CJ}(w.\lambda)\otimes_{\tS^\vee}{\tQ^{\vee}})^x=(\Gamma(\CJ^\prime, \SB^\vee(w))\otimes_{S^\vee}{\tQ^{\vee}})^x
$$
for all $x\in\CJ^\prime$.
Taking everything together we get
\begin{align*}
(P^{\CJ}(w.\lambda):\Delta(x.\lambda))&=(\tP^{\CJ}(w.\lambda):\tDelta(x.\lambda))\\
&=\dim_Q \tilde\DV^\CJ(\tP^{\CJ}(w.\lambda)\otimes_{\tilde S^\vee}\tilde Q^\vee)^x\\
&=(\Gamma(\CJ^\prime, \SB^\vee(w))\otimes_{S^\vee}{\tQ^{\vee}})^x\\
&=\rk\, \SB^\vee(w)^x.
\end{align*}
For the last step we refer to the characterization of Braden--MacPherson sheaves. 
\end{proof}

\subsection{A proof of Conjecture \ref{conj-KL}}

We can now collect all results and give a proof of the antidominant case of the Kazhdan--Lusztig conjectures. Note that the hardest part in the proof below probably is the result in Theorem \ref{thm-KLforIC}, that we quote. 

\begin{theorem} Conjecture \ref{conj-KL} is true.
\end{theorem}

\begin{proof} As already mentioned, Conjecture \ref{conj-KL} is equivalent to Conjecture \ref{conj-KLmult}, which is really just stating the inverse formula. By Theorem \ref{thm-BGG} we have
$$
[\Delta(x.\lambda):L(y.\lambda)]=(P^\CJ(y.\lambda):\Delta(x.\lambda)),
$$
where $\CJ\subset \fhd$ is open, locally bounded, and contains $x.\lambda$ and $y.\lambda$. Theorem \ref{thm-repmomgra} now shows that 
$$
(P^\CJ(y.\lambda):\Delta(x.\lambda))=\rk_S\,\SB^\vee(y)^x.
$$
Theorem \ref{thm-ICBMP} (together with Remark \ref{rem-degfromglobsec}) yields
$$
\rk_S\,\SB^\vee(y)^x=\rk_S\,\DH_T(i^\ast_x\SF),
$$
where $\SF$ is the equivariant IC-sheaf on $Y^y(\CJ)$. Now $\rk_S\,\DH_T(i^\ast_x\SF)=\dim_\DC \DH_T(i^\ast_x \IC_{\CJ,y})$ by formality. Finally, by Theorem \ref{thm-KLforIC}, 
$$
\dim_\DC \DH_T(i_x^\ast\IC_{\CJ,y})=Q_{x,y}(1).
$$
\end{proof}
 
\end{document}